\documentclass[12pt]{amsart}
\usepackage{enumerate}
\usepackage{color}
\usepackage{epsfig}
\usepackage{amsmath,amssymb}             
\usepackage{amsfonts,amsthm,latexsym}

\newtheorem{Thm}{Theorem}[section]
\newtheorem{Lem}[Thm]{Lemma}

\theoremstyle{remark}
\newtheorem{Example}[Thm]{Example}
\newtheorem{Rem}[Thm]{Remark}

\numberwithin{equation}{section}

\newcommand{\T}{\mathbb{T}}
\newcommand{\C}{\mathbb{C}}

\newcommand{\R}{\mathbb{R}}
\newcommand{\D}{\mathbb{D}}

\newcommand{\HH}{\mathcal{H}}

\newcommand{\DD}{\mathcal{D}}

\begin{document}

\title[Carleson Measures and the Reproducing Kernel Thesis]
 {Carleson Measures and Reproducing Kernel Thesis in Dirichlet-type spaces}

\author[G.R. Chac\'{o}n]{Gerardo R. Chac\'{o}n}
\author[E. Fricain]{Emmanuel Fricain}
\author[M. Shabankhah]{Mahmood Shabankhah}
\address{Gerardo R. Chac\'{o}n, Departamento de Matematicas, Pontificia Universidad Javeriana, Bogot\'{a}, Colombia}
\email{chacong@javeriana.edu.co}
\address{Emmanuel Fricain, Universit\'e de Lyon; Universit\'e Lyon 1; Institut Camille Jordan CNRS UMR 5208;
43, boulevard du 11 Novembre 1918, F-69622 Villeurbanne}
\email{fricain@math.univ-lyon1.fr}
\address{Mahmood Shabankhah, Laboratoire Paul Painlev\'e, UMR 8524,
Universit\'e Lille 1, 59655 Villeneuve d'Ascq Cedex,
France.}
\address{{\it Current address: }Department of Mathematics and Statistics\\ 
McGill University\\ Montreal, QC\\Canada H3A 2K6}
\email{mshaban@math.mcgill.ca}
\email{shabankh@math.univ-lille1.fr}
\thanks{The second author is partially supported by the ANR FRAB.
The third author's research is supported by ANR DYNOP and FQRNT}

\subjclass[2010]{Primary: 46E22; Secondary: 30H10, 31C25}



\keywords{Dirichlet-type spaces, Carleson measures, reproducing kernel thesis}

\begin{abstract}
In this paper, using a  generalization of a Richter and Sundberg
representation theorem, we give a new characterization of Carleson measures for the Dirichlet-type space $\DD(\mu)$ when $\mu$ is a finite sum of point masses.
A reproducing kernel thesis result is also
established in this case.
\end{abstract}

\maketitle

\section{Introduction}
Dirichlet-type spaces (also called local Dirichlet spaces) have been introduced by
S. Richter \cite{Rich91} when investigating analytic two-isometries. This class
of operators appeared for the first time in \cite{agler} in connection with
the compression of a first-order differential operator to the Hardy space $H^2$
of the unit disc. The study of two-isometries and related operators is also of interest
for its relations with the theory of dilations and
invariant subspaces of the shift operator on the classical Dirichlet space $\DD$ \cite{Richter-JReine}.
It is an immediate consequence of
the norm definitions that $M_z$, the operator of multiplication by the independant
variable $z$, is an isometry on $H^2$ but not on $\DD$ (see Section 2 for precise definitions). In fact, one can
verify that $M_z$ is an analytic two isometry on $\DD$. It is a remarkable result of
S. Richter \cite{Rich91} that every analytic two-isometry satisfying
$\textrm{dim Ker}(T^*)=1$ is unitarily equivalent to $M_z$ on some
Dirichlet-type space $\DD(\mu)$. These spaces have been studied ever since by several authors, see for example
\cite{A93}, \cite{GR1}, \cite{GR2}, \cite{Cha1}, \cite{Chartrand}, \cite{RS92},
\cite{Sar97}, \cite{Sa2}, \cite{Ser}, \cite{Shi01} and \cite{Shi}.

In particular, in \cite{GR1}, the first author introduces a notion of capacity adapted
to Dirichlet-type spaces and he gives a characterization of Carleson measures for
$\DD(\mu)$ in terms of this capacity. Carleson measures for the Hardy space
have proved to be objects of fundamental importance in the development of modern function
theory. In particular, they have appeared in areas ranging from the celebrated Corona problem and its solution by Carleson \cite{Car62}, to the
development of bounded mean oscillation (BMO) functions by C. Fefferman and E. Stein \cite{FefStein72},
P. Jones \cite{Jones80} and many others. The characterization obtained in \cite{GR1} is
similar to those given by D.~Stegenga in \cite{stegenga} for the classical Dirichlet space.
In the present paper, we will provide a new characterization of the Carleson measures for
the space $\DD(\mu)$ when $\mu$ is a finite sum of point masses. As we will see,
in this case, Carleson measures for $\DD(\mu)$ are determined, in a very specific way,
from those of $H^2$. The key idea is a generalization of Richter-Sundberg
representation theorem.

The other natural question we address is the reproducing kernel thesis for the embedding
$\DD(\mu)\hookrightarrow L^2(\nu)$, where $\nu$ is a positive Borel measure on the unit
disc. Recall that an operator on a reproducing kernel Hilbert space is said to satisfy the
Reproducing Kernel Thesis (RKT) if its boundedness is determined by its behaviour on the
reproducing kernels. In general, there is no reason why this should be true
but it turns out, as was proved by L. Carleson, that this is indeed the case
for the identity map $I: H^2 \to L^2(\nu)$. More explicitly, $I$ is bounded
(compact, respectively) on $H^2$ if and only if it acts as a
bounded (compact, respectively) operator on the set $\{k_z, z\in \D\}$.
Though there were many results of this type since Carleson's result, philosophically
the idea to study (RKT) for classes of operators in general reproducing kernel Hilbert
spaces comes from \cite{Havin-Nikolski} (see also \cite{Nikolski-controle}). We will show
that the identity map $I:\DD(\mu)\to L^2(\nu)$ is another example of operators
satisfying the (RKT), in the case where $\mu$ is a finite sum of point masses. Let us mention that there is another natural generalization of the classical
Dirichlet space $\DD$, the so-called weighted Dirichlet spaces, where the (RKT)
for the embedding in $L^2(\nu)$ space is not valid (see Remark~\ref{Rem:weighted-dirichlet-spaces}
for further details). 

The plan of the paper is the following. The next section contains preliminary material
concerning Dirichlet-type spaces and Carleson measures. Section 3 contains a
representation theorem for functions in $\DD(\mu)$ spaces corresponding to the case where
$\mu$ is a finite sum of point masses. In Section 4, we give a new characterization of
Carleson measures in Dirichlet-type spaces induced by finetely atomic measures. In sections 5 and 6  a reproducing kernel thesis for the embedding
$\DD(\mu)\hookrightarrow L^2(\nu)$ is established. Finally, in section 7, compact Carleson measures for Dirichlet-type spaces induced by finitely atomic measures are characterized in terms of the normalized reproducing kernels of the space.

\section{Preliminaries}
Recall that the Hardy space $H^2$ of the unit disc $\D=\{z\in\C:|z|<1\}$ is the Hilbert
space of functions $f(z)=\sum_{n\geq 0}a_n z^n$ analytic on $\D$ and such that
\[
\|f\|_2^2:=\sum_{n\geq 0}|a_n|^2<+\infty.
\]
If $f$ is an analytic function on $\D$, then we denote by $\DD(f)$ its Dirichlet integral
given by
\[
\DD(f) := \int_\D |f'(z)|^2 \, dA(z),
\]
where $dA$ is the normalized area measure.
Then the classical Dirichlet space $\DD$ consists of analytic functions
whose Dirichlet integral is finite. It is easily verified that $\DD$ is a
subspace of $H^2$. In particular, $\DD$ turns into a Hilbert space under the following natural norm
\[
\|f\|_\DD^2 := \|f\|_2^2 + \DD(f).
\]

Now let $\mu$ be a positive finite Borel measure on the unit circle $\T$, and let $P_\mu$
be its harmonic extension to $\D$, i.e.,
\[
P_\mu(z) = \int_\T \frac{1-|z|^2}{|\zeta-z|^2}\, d\mu(\zeta), \hspace{1cm} (z\in \D).
\]
The corresponding $\DD(\mu)$ space is defined to be the set of holomorphic functions
$f$  for which
\[
\DD_\mu(f) := \int_\D |f'(z)|^2 \,P_\mu(z)\, dA(z) < \infty.
\]
In particular, if $\mu$ is taken to be the normalized Lebesgue measure on $\T$, then  $P_\mu(z)=1$, $z\in\D$, and
therefore $\DD(\mu)=\DD$, the classical
Dirichlet space.
The quantity $\DD_\mu(f)$ is called the Dirichlet integral of $f$ with respect to $\mu$.
If $\mu=\delta_\lambda$, i.e., the unit point mass at $\lambda\in \T$, one writes $\DD_\lambda(f)$
instead of $\DD_\mu(f)$, and calls it the local Dirichlet integral of $f$ at $\lambda$.
By Fubini's Theorem, we have
\begin{equation}\label{lint}
\DD_\mu(f) = \int_\T \DD_\lambda(f) \, d\mu(\lambda).
\end{equation}

In \cite{RichSun91}, S. Richter and C. Sundberg proved that $f \in \DD(\delta_\lambda)$ if and only if $f=c+(z-\lambda)g$,
where $c$ is a constant and $g\in H^2$. In this case, $c$ is the non-tangential (even oricyclic) limit
of $f$ at $\lambda$, denoted by $f(\lambda)$, and $\|g\|_2^2 = \DD_\lambda(f)$. In other words,
\begin{equation}\label{eq:local-direct-integral}
\DD_\lambda(f) = \Bigl\| \frac{f-f(\lambda)}{z-\lambda}\Bigr\|_2^2.
\end{equation}
A different proof of these facts can be found in \cite{Sar97}.
Since $\DD(\delta_\lambda) \subset H^2$, it follows from (\ref{lint}) that $\DD(\mu) \subset H^2$.
The norm, with respect to which $\DD(\mu)$ is a Hilbert space, is given by
\begin{eqnarray}\label{eq:definition-norme-dirichlet-local}
\|f\|_\mu^2 := \|f\|_2^2+\DD_\mu(f).
\end{eqnarray}
Moreover, polynomials form a dense subset of $\DD(\mu)$ \cite{Rich91}.

The Hardy space $H^2$ is an example of a reproducing kernel Hilbert space,
i.e., the point evaluation $f \mapsto f(z)$ is a bounded functional on $H^2$,
for every $z \in \D$. In particular,
\[
f(z) = \langle f,k_z \rangle_2,
\]
for some $k_z \in H^2$. One calls $k_z$ the reproducing kernel at $z$.
According to \eqref{eq:definition-norme-dirichlet-local}, the point evaluation
$f \mapsto f(z)$ is also a bounded functional on $\DD(\mu)$, for every $z\in\D$. Now, if $\lambda\in\T$,
using the uniform boundedness principle, we easily see that the functional $f\mapsto f(\lambda)$ is bounded on
$\DD(\delta_\lambda)$ (whereas it is never bounded on $H^2$).

An important notion in the theory of complex function spaces is that of Carleson measures.
Given a Banach space $X$ of holomorphic functions in $\D$,
one says that $\nu$, a positive Borel measure on $\D$, is a {\em Carleson measure for $X$}
if the identity map $I: X \to L^2(\nu)$ is a bounded operator,
i.e., there is a constant $C>0$ such that
\begin{equation}\label{cm}
\int_\D |f|^2 \, d\nu \, \leq C \|f\|_X^2,
\end{equation}
for every function $f$ in $X$.
If the embedding happens to be compact, then $\nu$ is called a compact (or vanishing) Carleson measure.

In the setting of the Hardy space $H^2$, a geometric characterization of the corresponding measures
was obtained by L.~Carleson. He proved that $\nu$ is a Carleson
measure for $H^2$ if and only if
\[
\nu(S(\zeta,h))=O(h), \hspace{1cm} (h \to 0^+),
\]
where, for $\zeta \in \T$ and $0<h<1$, the set $S(\zeta,h)$ is given by
\begin{equation}\label{Cwin}
S(\zeta,h):=\{z\in \D:  1-h<|z|<1 \mbox{~and~} \Bigl|\frac{z}{|z|}-\zeta \Bigr|<\frac{h}{2}\}.
\end{equation}
See for example \cite{Gar07} for a proof of this classical result.
As one might expect, the characterization of compact Carleson measures for $H^2$
is obtained by replacing ``O'' in the above condition by ``o''.

\section{A representation Theorem}

The following result extends Richter-Sundberg representation theorem \cite{RichSun91}
to the case of finitely atomic measures. A part of this result already appeared
(at least implicitly) in \cite{Sa2}.

\begin{Thm}\label{Thm-rep}
Let $\mu = \sum_{j=1}^n \alpha_j \delta_{\lambda_j}$ be a finitely atomic measure, $\alpha_j>0$, $\lambda_j\in\T$,
$1\leq j\leq n$.
Then $f \in \DD(\mu)$ if and only if there exist $g \in H^2$ and a polynomial $p$, $\deg p \leq n-1$, such that
\begin{equation}\label{rep}
f = p + \prod_{j=1}^n (z-\lambda_j) g.
\end{equation}
Moreover, $g$ and $p$ are unique, and we have
\begin{equation}\label{norm-ineq}
\|g\|_2 \leq C \|f\|_\mu,
\end{equation}
for some positive constant $C=C(n,\mu)$.
\end{Thm}

\begin{proof}
If $f$ is given by (\ref{rep}), then $f \in \DD(\mu)$. Indeed, we have
\[
\prod_{j=1}^n (z-\lambda_j) g \in \bigcap_{j=1}^n \DD(\delta_{\lambda_j}) = \DD(\mu).
\]

Now let $f \in \DD(\mu)$. We will show that $f$ has a representation of form (\ref{rep}).
The case $n=1$ being already established in \cite{RichSun91}, we consider first the case $n=2$.
Let $f\in \DD(\mu)$. Suppose that $f(\lambda_1) = f(\lambda_2) = 0$.
Since $f\in \DD(\delta_{\lambda_1})\cap\DD(\delta_{\lambda_2})$, we have
\[
f = (z-\lambda_1) g_1 = (z-\lambda_2) g_2,
\]
for some $g_1, g_2 \in H^2$. A simple calculation shows that $g_1/(z-\lambda_2) \in H^2$.
In particular, $f$ can be written as
\[
f = (z-\lambda_1) (z-\lambda_2) g,
\]
where $g \in H^2$.
Using induction, we get that if $f \in \DD(\mu)$, with $f(\lambda_j)=0$, $1 \leq j \leq n$, then
\[
f = (z-\lambda_1) \cdots (z-\lambda_n) g,
\]
for some $g\in H^2$. For the general case, we choose a polynomial $p$, of degree less than or equal to $n-1$,
such that
\[
p(\lambda_j)=f(\lambda_j), \hspace{1cm} (1\leq j \leq n).
\]
Upon applying the preceding argument to $f-p$, we obtain (\ref{rep}).

For the uniqueness in the decomposition \eqref{rep}, we just note that the polynomial $p$ is necessarily given by
\begin{equation}\label{pform}
p(z) = \sum_{j=1}^n  \prod_{1\leq k \leq n,\,k\neq j} \frac{z-\lambda_k}{\lambda_j-\lambda_k} \, f(\lambda_j).
\end{equation}

It remains to prove the norm inequality (\ref{norm-ineq}). Since the evaluation functionals
$f \mapsto f(\lambda_j)$ are bounded, $1 \leq j \leq n$,
we see that
\begin{eqnarray*}
\|p\|_\mu^2 & = & \|p\|_2^2 + \int_\D |p'(z)|^2 \, P_\mu(z) \, dA(z)\\
\\
& \leq & C \|f\|_\mu^2,
\end{eqnarray*}
where $C$ is a constant depending on $n$ and $\mu$.
In particular, by the triangle inequality,
\[
\Bigl\| \prod_{j=1}^n (z-\lambda_j) g \Bigr\|_\mu \leq C \|f\|_\mu.
\]

To complete the proof of (\ref{norm-ineq}) we need to show that
\[
\|g\|_2 \leq C \Bigl\| \prod_{j=1}^n (z-\lambda_j) g \Bigr\|_\mu.
\]
This is obvious if $n=1$. In fact, given $\mu=\alpha_1 \delta_{\lambda_1}$,
we have by \eqref{lint} and \eqref{eq:definition-norme-dirichlet-local},
\begin{eqnarray*}
\|(z-\lambda_1)g\|_\mu^2 & = & \|(z-\lambda_1)g\|_2^2 + \DD_\mu((z-\lambda_1)g)\\
& = & \|(z-\lambda_1)g\|_2^2 +\alpha_1 \DD_{\lambda_1}((z-\lambda_1)g)\\
& = & \|(z-\lambda_1)g\|_2^2 +\alpha_1 \|g\|_2^2\,\geq \alpha_1 \|g\|_2^2.
\end{eqnarray*}

For $n=2$, we have
\begin{eqnarray*}
|\lambda_1-\lambda_2|^2 \|g\|_2^2 & \leq & 2 \,\Bigr(\|(z-\lambda_1) g\|_2^2 + \|(z-\lambda_2)g\|_2^2 \Bigl)\\
\\
& \leq & 2 \,\Bigl\| \prod_{j=1}^2 (z-\lambda_j) g \Bigr\|_\mu^2.
\end{eqnarray*}
The general case follows by induction.
\end{proof}

\section{Carleson measures}

Theorem \ref{Thm-rep} enables us to give a complete description of Carleson measures
in Dirichlet spaces induced by finitely atomic measures.

\begin{Thm}\label{Car-m}
Let $\mu = \sum_{j=1}^n \alpha_j \delta_{\lambda_j}$ be a finitely atomic measure,
$\alpha_j>0$, $\lambda_j\in\T$, $1\leq j\leq n$.
A finite positive Borel measure $\nu$ on $\D$ is a Carleson measure
for $\DD(\mu)$
if and only if $\prod_{j=1}^n |z-\lambda_j|^2 d\nu(z)$ is a Carleson measure for $H^2$.
\end{Thm}

\begin{proof}
Assume that $\nu$ is a Carleson measure for $\DD(\mu)$.
Since for every $g\in H^2$,  $f = \prod_{j=1}^n (z-\lambda_j) g \in \DD(\mu)$ , we have
\[
\int_\D \prod_{j=1}^n |z-\lambda_j|^2 |g|^2\, d\nu(z) \leq C \|f\|_\mu^2 \leq C' \|g\|_2^2,
\]
where the last inequality follows easily from the definition of norm in $\DD(\mu)$.
Consequently, $\prod_{j=1}^n |z-\lambda_j|^2 d\nu(z)$ is a Carleson measure for $H^2$.

Conversely, if $\prod_{j=1}^n |z-\lambda_j|^2 d\nu(z)$ is a Carleson measure for $H^2$, we will show
that $\nu$ is a Carleson measure for $\DD(\mu)$.
Let $f \in \DD(\mu)$. By Theorem \ref{Thm-rep},
\[
f = p + \prod_{j=1}^n (z-\lambda_j) g,
\]
where $g \in H^2$ and $p$ is given by (\ref{pform}).
Since the evaluation functionals $f \mapsto f(\lambda_j)$, $1\leq j \leq n$, are bounded on $\DD(\mu)$,
we see that
\begin{eqnarray*}
\int_\D |f|^2 \, d\nu & \leq & C \, \Bigl( \int_\D |p|^2 \, d\nu + \int_\D \prod_{j=1}^n |z-\lambda_j|^2 |g|^2 \,
d\nu \Bigr)\\
\\
& \leq & C \, (\|f\|_\mu^2 + \|g\|_2^2).
\end{eqnarray*}
An appeal to (\ref{norm-ineq}) yields the desired conclusion.
\end{proof}

\begin{Example}\label{exemple-mesure-carleson}
It is clear that every Carleson measure for $H^2$ is a Carleson measure for $\DD(\mu)$ as well.
The converse, however, is not true.
Let $\alpha \in (0,1)$, and let $d\nu(z)=(1-|z|)^{-\alpha} dm(z)$, where $dm$ is the Lebesgue measure on $[0,1)$.
Since $dm$ is a Carleson measure for $H^2$, the preceding result implies that $d\nu$ is indeed a
Carleson measure for $\DD(\delta_1)$. We claim that $d\nu$ is not a Carleson measure for $H^2$.
For $\zeta \in \T$ and $0<h<1$, consider $S(\zeta,h)$ as given by (\ref{Cwin}).
If $S(\zeta,h) \bigcap [0,1) \ne \emptyset$, then
\[
\nu\Bigl(S(\zeta,h)\Bigr)  = \int_{1-h}^1 \frac{dr}{(1-r)^\alpha} = \frac{h^{1-\alpha}}{1-\alpha},
\]
which proves the claim because
\[
\lim_{h \to 0^+} \frac{\nu\Bigl(S(\zeta,h)\Bigr)}{h} = \infty.
\]

In fact, given $0<\beta<1$, if we choose $\alpha$ such that $1-\beta<\alpha<1$, then the same example
gives a Carleson measure $\nu$ for $\DD(\delta_1)$ which satisfies
$\nu(S(\zeta,h))/h^\beta\to +\infty$, as $h\to  0^+$.
That means that we cannot translate directly our characterization of Carleson measures for $\DD(\mu)$
in terms of the asymptotic of $\nu(S(\zeta,h))$.
\end{Example}

As the following result shows, compact Carleson measures for $\DD(\mu)$ are characterized in a similar manner.

\begin{Thm}\label{c-Car-m}
Let $\mu = \sum_{j=1}^n \alpha_j \delta_{\lambda_j}$ be a finitely atomic measure.
A finite positive Borel measure $\nu$ is a compact Carleson measure
for $\DD(\mu)$
if and only if $\prod_{j=1}^n |z-\lambda_j|^2 \, d\nu(z)$ is a compact Carleson measure for $H^2$.
\end{Thm}

\begin{proof}
Suppose that $\nu$ is a compact Carleson measure for $\DD(\mu)$.
Let $(g_k)_{k \geq 1} \subset H^2$ such that $(g_k)_k$ converges weakly to zero,
as $k \to \infty$. Put $f_k= \prod_{j=1}^n (z-\lambda_j) g_k$, $k \geq 1$.
Since the operator $T: H^2 \to \DD(\mu)$ of multiplication by
$\prod_{j=1}^n (z-\lambda_j)$ is bounded, it follows that $f_k \to 0$ weakly in $\DD(\mu)$.
So, the assumption that $I: \DD(\mu) \to L^2(\nu)$ is compact
implies that
\[
\int_{\D} |f_k|^2 \, d\nu(z) = \int_\D |g_k|^2 \prod_{j=1}^n |z-\lambda_j|^2 \, d\nu(z)
\to 0, \hspace{1cm} (k \to \infty).
\]
Therefore, $\prod_{j=1}^n |z-\lambda_j|^2\,d\nu(z)$ is a compact Carleson measure for $H^2$.

For the converse, assume that $(f_k)_k$ is a sequence in $\DD(\mu)$ which converges weakly to zero.
We will show that $\|f_k\|_{L^2(\nu)} \to 0$, as $k \to \infty$.
Indeed, by Theorem (\ref{Thm-rep}), $f_k= p_k + \prod_{j=1}^n (z-\lambda_j) g_k$, where $g_k \in H^2$ and
\[
p_k(z) = \sum_{j=1}^n  \prod_{m\neq j} \frac{z-\lambda_m}{\lambda_j-\lambda_m} \, f_k(\lambda_j),
\hspace{1cm} (k\geq 1).
\]
By (\ref{norm-ineq}), we see that $(g_k)_k$ is a bounded sequence in $H^2$.
Moreover, since the evaluation functionals $f \mapsto f(\lambda_j)$ are bounded on $\DD(\mu)$,
for $1 \leq j \leq n$,
and $f_k$ converges to zero uniformly on compact subsets of $\D$, we see that
$g_k \to 0$ uniformly on compacts.
In other words, $(g_k)_k$ is a sequence converging weakly to zero in $H^2$.
Hence,
\begin{eqnarray*}
\int_\D |f_k|^2 \, d\nu & \leq & C \Bigl( \int_\D |p_k(z)|^2 \, d\nu(z) + \int_\D |g_k|^2 \,
\prod_{j=1}^n |z-\lambda_j|^2 \, d\nu(z) \Bigr)\\
\\
& \leq & C \Bigl( \sum_{j=1}^n |f_k(\lambda_j)|^2 + \int_\D |g_k|^2 \,
\prod_{j=1}^n |z-\lambda_j|^2 \, d\nu(z) \Bigr)\, \to \, 0,
\end{eqnarray*}
as $k \to \infty$. The proof is now complete.

\end{proof}

\begin{Rem}
In \cite{Chartrand}, R. Chartrand introduces a notion of a Carleson type measure in a different way than the one
used here and in \cite{GR1}. In fact, in \cite{GR1}, exhibiting two examples, it is proved that the two definitions
are really differents.
\end{Rem}

\section{Reproducing kernel thesis for one point mass}

Dirichlet-type spaces $\DD(\mu)$ are reproducing kernel Hilbert spaces.
Here, we will prove a reproducing kernel thesis result in the case when $\mu$ is a finite sum of point masses
(Theorem~\ref{thesis-thm} and Theorem~\ref{thm:rkt-finite}). One of the difficulties here is to
obtain an explicit expression for the reproducing kernels. As far as we know, except for the case of
$\DD(\delta_\lambda)$, such an explicit expression is not known.

Improving an earlier result of Richter and Sundberg \cite{RichSun91},
Sarason \cite{Sar97} showed that if $\mu = \delta_\lambda$,
then $\DD(\mu)$ can be identified with a de Branges-Rovnyak space.
More precisely, Sarason proved that $\DD(\delta_\lambda)=\HH(b_\lambda)$, with the equality of norms, where
\begin{equation}\label{kernel}
b_\lambda(z)= \frac{(1-a_0) \overline{\lambda} z}{1-a_0 \overline{\lambda} z}, \hspace{1cm} (z \in \D),
\end{equation}
where $a_0$ is the smallest root of $(a_0-1)^2=a_0$. Recall here that, given $b$ in the closed unit ball of
$H^\infty$, the de Branges--Rovnyak space $\HH(b)$ is defined as the range of the operator $(I-T_bT_b^*)^{1/2}$,
equipped with the range norm, where $T_b$ denotes the Toeplitz operator on $H^2$
(which, in this case, corresponds to the operator of multiplication  by $b$ on $H^2$).
Then we know \cite{sarason-debr} that $\HH(b)$ is a reproducing kernel Hilbert space, whose reproducing kernel
functions are given by
\[
k_w^b(z)=\frac{1-\overline{b(w)} b(z)}{1-\overline{w}z}, \hspace{1cm} (w,z \in \D).
\]
Now if $\mu=\alpha \delta_\lambda$, with $\alpha>0$, then one still has $\DD(\mu)=\HH(b_\lambda)$, except that
$a_0$ in \eqref{kernel} has to be replaced by the smallest root of the equation $(a_0-1)^2=\alpha a_0$ instead.
This is in fact proved in \cite{Ransford-Guillot-Chevrot} where it is also shown that this is the only case
where $\DD(\mu)$ arises as a de Branges--Rovnyak space (see also \cite{GR2}).

Using this identification, we first prove the following reproducing kernel thesis in the special case of
a point mass.

\begin{Thm}\label{thesis-thm}
Let $\nu$ be a finite positive Borel measure on $\D$, and let $\mu=\alpha \delta_\lambda$, where $\lambda\in\T$
and $\alpha>0$. Then $\nu$ is a Carleson measure for $\DD(\mu)$ if and only if there exists a
constant $C>0$ such that
\begin{equation}\label{C-thesis-1}
\int_{\mathbb{D}}
|k_w^\mu|^2d\nu \leq C \|k_w^\mu\|_{\mu}^2,
\end{equation}
for every $w\in\D$, where, $k_w^\mu$ is the reproducing kernel for $\DD(\mu)$ at $w$.
\end{Thm}

\begin{proof}
Suppose that (\ref{C-thesis-1}) holds. By the paragraph preceding the theo\-rem,
we have
\begin{equation}\label{C-thesis-2}
\int_{\D} \Bigl| \frac{1-\overline{b_\lambda(w)} b_\lambda(z)}{1-\overline{w} z} \Bigr|^2\, d\nu(z)
\,\leq\, C \frac{1-|b_\lambda(w)|^2}{1-|w|^2}, \hspace{1cm} (w \in \D),
\end{equation}
where $b_\lambda$ is given by (\ref{kernel}) and $a_0 \in (0,1)$ satisfies $(a_0-1)^2=\alpha a_0$.
We will show that (\ref{C-thesis-2}) implies
\begin{equation}\label{C-thesis-3}
\int_{\D} \frac{1}{|1-\overline{w}z|^2} \, |z-\lambda|^2 \, d\nu(z) \leq C_1 \, \frac{1}{1-|w|^2}, \hspace{1cm}
(w \in \D),
\end{equation}
for some positive constant $C_1$.
Inequality (\ref{C-thesis-3}), in conjunction with Carleson's reproducing kernel thesis for $H^2$, implies that
$|z-\lambda|^2 d\nu(z)$ is a Carleson measure for $H^2$. The desired conclusion then follows from
Theorem \ref{Car-m}.

To prove (\ref{C-thesis-3}), it is sufficient to show that
\begin{equation}\label{eq:rkt-cle-one-point}
\frac{1-|w|^2}{|1-\overline{w}z|^2} \, |z-\lambda|^2 \,\leq\, C_2 \,
\Bigl| \frac{1-\overline{b_\lambda(w)} b_\lambda(z)}{1-\overline{w} z} \Bigr|^2\,
\frac{1-|w|^2}{1-|b_\lambda(w)|^2}, \qquad (z,w \in \D),
\end{equation}
or equivalently
\begin{equation}\label{inf-1}
\inf_{z,w \in \D} \frac{\Bigl| 1-\overline{b_\lambda(w)} b_\lambda(z)
\Bigr|^2}{(1-|b_\lambda(w)|^2)|z-\lambda|^2} \,>\, 0.
\end{equation}
Using the inequality
\[
\frac{(1-|\gamma|^2)(1-\beta|^2)}{|1-\overline{\gamma}\beta|^2} =
1-\Bigl|\frac{\gamma-\beta}{1-\overline{\gamma}\beta}\Bigr|^2 \leq 1,
\hspace{1cm} (\gamma,\beta \in \D),
\]
we see that (\ref{inf-1}) is proved as soon as we show that
\begin{equation}\label{inf-2}
\inf_{z\in \D} \frac{1-|b_\lambda(z)|^2}{|z-\lambda|^2} \,>\, 0.
\end{equation}
Note that condition \eqref{inf-2} follows immediately from
\begin{equation}\label{inf-3}
1-|b_\lambda(z)|^2\geq \frac{a_0|z-\lambda|^2}{|\lambda-a_0z|^2},\qquad z\in\D.
\end{equation}
To prove \eqref{inf-3}, we can of course assume that $\lambda=1$, and then it reduces to
\begin{equation}\label{inf-4}
1-|b_1(z)|^2\geq \frac{a_0|z-1|^2}{|1-a_0z|^2},\qquad z\in\D.
\end{equation}
But using the identity $(1-a_0)^2=\alpha a_0$, we have
\begin{eqnarray*}
1-|b_1(z)|^2&=&1-\frac{\alpha a_0|z|^2}{|1-a_0z|^2}=\frac{|1-a_0z|^2-\alpha a_0|z|^2}{|1-a_0z^2|}\\
&=&\frac{1+a_0^2 |z|^2-2a_0\Re(z)- \alpha a_0 |z|^2}{|1-a_0z|^2}.
\end{eqnarray*}
Therefore \eqref{inf-4} is equivalent to
\[
1+a_0^2 |z|^2-2a_0\Re(z)-\alpha a_0 |z|^2\geq a_0(1+|z|^2-2\Re(z))
\]
or
\[
1-a_0+a_0(a_0-\alpha-1)|z|^2\geq 0.
\]
Using now the relation $a_0(a_0-\alpha-1)=a_0-1$, we obtain immediately the last inequality for every $z\in\D$.
That proves \eqref{inf-4} and the result.
\end{proof}

\begin{Rem}
In \cite{Ba-Fr-Ma}, a reproducing kernel thesis is proved for the embedding of $\HH(b)$ into $L^2(\nu)$
in the case where the function $b$ satisfied some additional hypothesis. More precisely, assume that $b$
satisfies the so-called connected level set condition (which means that for some $\varepsilon\in (0,1)$,
the level set $\Omega(b,\varepsilon):=\{z\in\D:|b(z)|<\varepsilon\}$ is connected) and assume further
that the spectrum of $b$, $\sigma(b)=\{\zeta\in\T:\liminf_{z\to\zeta}|b(z)|<1\}$, is contained in the
closure of $\Omega(b,\varepsilon)$. Then the operator $f\mapsto f$ is bounded from $\HH(b)$ into $L^2(\nu)$
if and only if it is bounded on reproducing kernels.

In our case (which corresponds to a Dirichlet-type space), the hypothesis
$\sigma(b_\lambda)\subset\hbox{Clos }\Omega(b_\lambda,\varepsilon)$ is not satisfied, whence we cannot apply
\cite[Theorem 6.8]{Ba-Fr-Ma} to get Theorem~\ref{thesis-thm}.
Indeed, we easily check that $\sigma(b_\lambda)=\T \setminus \{\lambda\}$;
now if $\sigma(b_\lambda)\subset\hbox{Clos }\Omega(b_\lambda,\varepsilon)$, then we would have
$\lambda\in\hbox{Clos }\Omega(b_\lambda,\varepsilon)$, but this is absurd by continuity of $b_\lambda$
on the closed unit disc.
\end{Rem}

%
%
%
%
%
%
%

\section{Reproducing kernel thesis for $\mu=\sum_{j=1}^n \alpha_j \delta_{\lambda_j}$}

As we have ever mentioned, in the case where $\mu$ is a finite sum of point masses, an explicit formula for the
reproducing kernels is not known (except for one point mass). The reproducing kernel thesis can however be shown
without such an explicit formula by using some techniques from \cite{GR2}. The main idea is to relate the
reproducing kernels in this case to the reproducing kernels in the case in which $\mu$ is an atomic measure and
then use Theorem~\ref{thesis-thm}.

\begin{Thm}\label{thm:rkt-finite}
Let $\nu$ be a finite positive Borel measure on $\D$, and let $\mu= \sum_{j=1}^n \alpha_j \delta_{\lambda_j}$,
$\lambda_j\in\T$, $\alpha_j>0$. Then $\nu$ is a Carleson measure for $\DD(\mu)$ if and only if there exists a
constant $C>0$ such that
\[
\int_{\mathbb{D}}
|k_w^\mu|^2d\nu \leq C \|k_w^\mu\|_{\mu}^2,
\] for every $w\in\D$, where, $k_w^\mu$ is the reproducing kernel for $\DD(\mu)$ at $w$.
\end{Thm}

%
%
Before proving this theorem, we need some preliminary results. First, we will need a general result about
complete Nevannlina-Pick reproducing kernels.
Recall that a reproducing kernel $k$ on the unit disc is a {\em complete Nevannlina-Pick kernel
(complete NP kernel)} if $k_0(z)=1$ for all $z\in\D$ and if there exists a sequence of analytic functions
$\{b_n\}_{n\geq 1}$ on $\D$ such that $$1-\frac{1}{k_\lambda (z)}=\sum_{n\geq 1} b_n(z)\overline{b_n(\lambda)},
\qquad\text{for all } \lambda, z\in \D.$$ This condition is equivalent to the assumption that $1-1/k$
is positive definite. Shimorin in \cite{Shi} showed that the $\DD(\mu)$ spaces have a complete NP kernel.
The first result we will need is due to McCullough and Trent \cite{MT}.
We will say that a subspace ${\mathcal M}$ of a Hilbert space $H$ is a {\em multiplier invariant subspace} if
$\varphi{\mathcal M}\subset {\mathcal M}$ for every $\varphi\in M(H)$, the space of multipliers of $H$.

\noindent {\bf Theorem A} (McCullough and Trent \cite{MT}){\bf .}
{\em Let $k$ be a complete NP kernel and let ${\mathcal M}$ be a multiplier invariant subspace.
Then there exists a sequence of multipliers $\{\varphi_n\}\subset{\mathcal M}$ such that
$$P_{\mathcal M}=\sum_{n\geq 1}M_{\varphi_n}M^\ast_{\varphi_n}\, (SOT)$$ where $P_{\mathcal M}$
denotes the projection onto ${\mathcal M}$ and $M_{\varphi_n}$ denotes the multiplication operator,
$f\mapsto \varphi_n f$, and the series converges in the strong operator topology.}


In particular, notice that if we take the function $k_z$, $z\in\D$, we have
$$P_{\mathcal M}k_z=\sum_{n\geq 1}M_{\varphi_n}M^\ast_{\varphi_n}k_z.$$
Since $M^\ast_{\varphi_n}k_z=\overline{\varphi_n(z)}k_z$, then, for every $w\in \D$, we have
$$
P_{\mathcal M}k_z(w)=\sum_{n\geq 1} \varphi_n (w)\overline{\varphi_n(z)}k_z(w),
$$
or equivalently,
\begin{equation}\label{eqposdef2}
\frac{P_{\mathcal M}k_z(w)}{k_z(w)}=\sum_{n\geq 1} \varphi_n (w)\overline{\varphi_n(z)},
\end{equation}
i.e., $\frac{P_{\mathcal M}k_z(w)}{k_z(w)}$ is positive definite.

We will also need the following result which is due to Richter and Sundberg \cite{RichSun91, RS92} and Aleman \cite{A93}.

\noindent {\bf Theorem B}{\bf .}
{\em
Let ${\mathcal M}$ be a multiplier invariant subspace of $\DD(\mu)$.
Then ${\rm dim}({\mathcal M}\ominus z{\mathcal M})=1$, and if $f\in {\mathcal M}\ominus z{\mathcal M}$,
with $\|f\|_{\mu}=1$,
then
\begin{itemize}
\item[\textup{(i)}] $|f(z)|\leq 1$ for all $z\in\D$.
\item[\textup{(ii)}] $\|fg\|_{\mu}=\|g\|_{\mu_f}$, for every $g\in \DD(\mu_f)$,
where $d\mu_f=|f|^2d\mu$.
\item[\textup{(iii)}]For every $g\in {\mathcal M}$, there exists $h\in \DD(\mu_f)$ such that $g=fh$.
\end{itemize}}
The following lemma which already appeared in \cite{GR2} is the key point in the proof of our theorem.

\begin{Lem}\label{lemposdef}
Let $\mu=\sum_{j=1}^n \alpha_j\delta_{\lambda_j}$, $\lambda_j\in\T$ and $\alpha_j>0$. Then, for every
$j=1,\dots,n$, there exists a positive constant $a_j$ such that if $\mu_j=a_j\delta_{\lambda_j}$, then
$\dfrac{k^{\mu_j}}{k^\mu}$ is positive definite.
\end{Lem}

\begin{proof}
First, notice that the kernel $k^\mu$ is never zero (see \cite{Shi01}) and
consequently the quotient is well defined. Let $j\in \{1,\dots, n\}$ be fixed and define
\[
{\mathcal M}_j:=\{f\in \DD(\mu):f(\lambda_i)=0, \, \forall i\neq j\}.
\]
Then ${\mathcal M}_j$
is a multiplier invariant subspace of $\DD(\mu)$. Take $\phi_j\in {\mathcal M}_j\ominus z{\mathcal M}_j$,
$\|\phi_j\|_{\mu}=1$. Then by Theorem B we see that the multiplication operator
$M_{\phi_j}:\DD(\mu_{\phi_j}) \to {\mathcal M}_j$ is an onto isometry (and consequently a unitary operator).
Here,
\[
d\mu_{\phi_j}=|\phi_j|^2d\mu=\sum_{i=1}^n\alpha_i|\phi_j(\lambda_i)|^2
\delta_{\lambda_i}=\alpha_j|\phi_j(\lambda_j)|^2 \delta_{\lambda_j},
\]
the last equality following from the fact that $\phi_j\in\mathcal M_j$. Define
$a_j:=\alpha_j|\phi_j(\lambda_j)|^2$ and $d\mu_j=d\mu_{\phi_j}$.
Then the reproducing kernel for the space ${\mathcal M}_j$ is given by
\[
k^{{\mathcal M}_j}_z(w)=\overline{\phi_j(z)}\phi_j(w) k^{\mu_j}_z(w).
\]
On the other hand, we also know that $k^{{\mathcal M}_j}_z=P_{{\mathcal M}_j}k^\mu_z$, hence
\begin{eqnarray*}
\frac{k^{\mu_j}_z(w)}{k^\mu_z(w)}&=&\frac{1}{\overline{\phi_j(z)}\phi_j(w)} \frac{k^{{\mathcal M}_j}_z(w)}{k^\mu_z(w)}\\
&=&\frac{1}{\overline{\phi_j(z)}\phi_j(w)} \frac{P_{{\mathcal M}_j}k^\mu_z(w)}{k^\mu_z(w)},
\end{eqnarray*}
and since each one of the factors is positive definite, then the result follows.
\end{proof}

Now we come to the proof of reproducing kernel thesis for the identity map
$I:\DD(\mu)\to L^2(\nu)$, in the case where $\mu$ is a finite sum of point
masses.

\noindent
{\it Proof of Theorem~\ref{thm:rkt-finite}.}$\quad$ Using Theorem~\ref{Car-m} and Carleson's reproducing
kernel thesis for $H^2$, it is sufficient to prove that if there exists a constant $C>0$ such that
\[
\int_\D |k^\mu_w|^2\,d\nu\leq C \|k_w^\mu\|^2_\mu,
\]
then
\[
\int_{\D}\frac{1}{|1-\bar{w}z|^2}\prod_{i=1}^n|z-\lambda_i|^2 d\nu(z)\lesssim \frac{1}{1-|w|^2}.
\]
Of course it is enough to show that
\begin{equation}\label{eq1}
\frac{1-|w|^2}{|1-\bar{w}z|^2}\prod_{i=1}^n |z-\lambda_i|^2 \lesssim \frac{|k_w^\mu(z)|^2}{\|k_w^\mu\|^2_{\mu}}.
\end{equation}
According to Lemma~\ref{lemposdef}, for every $j=1,\dots,n$, there exists a positive constant $a_j$ such that if
$\mu_j=a_j\delta_{\lambda_j}$, then $\dfrac{k^{\mu_j}}{k^\mu}$ is a positive definite function. In particular, that
implies
 \begin{equation}\label{eqposdef}
\left|\frac{k^{\mu_j}_z(w)}{k^\mu_z(w)}\right|^2\leq \frac{\|k^{\mu_j}_z\|_{\mu_j}^2\|k^{\mu_j}_w\|_{\mu_j}^2}
{\|k^\mu_z\|_{\mu}^2\|k^\mu_w\|_{\mu}^2}.
\end{equation}
Now recall that by \eqref{eq:rkt-cle-one-point}, there exists a constant $C>0$ such that, for every $j$ we have
\[
\frac{1-|w|^2}{|1-\bar{w}z|^2}|z-\lambda_j|^2\leq C \frac{|k_w^{\mu_j}(z)|^2}{\|k_w^{\mu_j}\|_{\mu_j}^2}.
\]
Therefore to show \eqref{eq1}, it is enough to establish the following inequality for every $j$:
\begin{equation}\label{eqenough2}
\prod_{i\neq j}|z-\lambda_i|^2\lesssim \frac{\|k_z^\mu\|^2_{\mu}}{\|k_z^{\mu_j}\|_{\mu_j}^2}.
\end{equation}
An equivalent equation  is the following one (obtained by using the notation of the proof of
Lemma \ref{lemposdef}):

\begin{equation}\label{eqenough3}
\|P_{\mathcal{M}_j}k_z^{\mu}\|^2_{\mu}\prod_{i\neq j}|z-\lambda_i|^2\lesssim |\phi_j(z)|^2 \|k_z^{\mu}\|^2_{\mu}.
\end{equation}

On the other hand, notice that since the multiplication operator $M_{\phi_j}$ is unitary,
then the operator $M_{\phi_j^{-1}}:\mathcal{M}_j\to \DD(\mu_j)$ is also unitary and satisfies that
$M_{\phi_j}^\ast=M_{\phi_j^{-1}}$. Therefore, if $f\in \mathcal{M}_j$ then
$M_{\phi^{-1}_j}f\in \DD(\mu_j)\subset H^2$.

Now, note that since $p_j(z):=\prod_{i\neq j}(z-\lambda_i)$ is a polynomial, then
$p_j\in H^\infty \cap \DD(\mu_{\phi_j})$ which means that $p_j$ is a multiplier of the space $\DD(\mu_j)$
and consequently $p_j M_{\phi_j^{-1}} f \in \DD(\mu_j)$ which implies that
$p_jM_{\phi_j^{-1}}f\in \DD(\delta_{\lambda_j})$.

If $i\neq j$, then since $f=\phi_j \tilde{f}$ for some $\tilde{f}\in \DD(\mu_j) \subset H^2$ we have by
\eqref{eq:local-direct-integral}
\begin{eqnarray*}
\mathcal D_{\lambda_i}(p_j\phi_j^{-1}f)&=&\frac{1}{2\pi}\int_0^{2\pi}\left|\frac{1}{e^{it}-\lambda_i}
\left(\frac{p_j(e^{it})f(e^{it})}{\phi_j(e^{it})}-\frac{p_j(\lambda_i)f(\lambda_i)}{\phi_j(\lambda_i)}\right)\right|^2dt\\
&=&\frac{1}{2\pi}\int_0^{2\pi}\left|\prod_{\ell\neq i, j} (e^{it}-\lambda_\ell)\tilde{f}(e^{it})\right|^2dt\\
&\lesssim& \|\tilde{f}\|_{H^2}^2<+\infty,
\end{eqnarray*}
and consequently $p_j\phi_j^{-1}f\in \DD(\delta_{\lambda_i})$. Thus $p_j\phi_j^{-1}f\in
\bigcap_{i=1}^n \DD(\delta_{\lambda_i})=\DD(\mu)$ and moreover $p_j\phi_j^{-1}f\in \mathcal{M}_j$.

Therefore, the multiplication operator $M_{p_j\phi_j^{-1}}:\mathcal{M}_j\to\mathcal{M}_j$
is well defined and by the closed-graph theorem we conclude that it is bounded. Hence,

\begin{eqnarray*}
\frac{|p_j(z)|\|k_z^{\mathcal{M}_j}\|^2_{\mu}}{|\phi_j(z)|} &=&
|\langle M_{p_j\phi_j^{-1}}(k_z^{\mathcal{M}_j}), k_z^\mu\rangle_{\mu}|\\
&\leq& \|M_{p_j\phi_j^{-1}} k_z^{\mathcal{M}_j}\|_{\mu}\|k_z^\mu\|_{\mu}\\
&\leq& \|M_{p_j\phi_j^{-1}}\|\|k_z^{\mathcal{M}_j}\|_{\mu}\|k_z^\mu\|_{\mu}
\end{eqnarray*}
where $\|M_{p_j\phi_j^{-1}}\|$ denotes the norm of the operator $M_{p_j\phi_j^{-1}}$ acting on $\mathcal{M}_j$.
This gives \eqref{eqenough3} and concludes the proof. \hfill$\square$

\medskip

\begin{Rem}\label{Rem:weighted-dirichlet-spaces}
Let $0\leq \alpha\leq 1$.
The weighted Dirichlet space $\DD_\alpha$ consists of analytic functions $f$ on $\D$ such that
\[
\|f\|_\alpha^2 := |f(0)|^2 + \int_\D |f'(z)|^2 (1-|z|^2)^\alpha\, dA(z) \,<\, \infty.
\]
Note that $\DD_1=H^2$ and $\DD_0=\DD$.
The reproducing kernels in $\DD_\alpha$ are given by
\[
k_w^\alpha(z) = \begin{cases}
                (1-\overline{w}z)^{-\alpha},&\hbox{for }0<\alpha\leq 1;\\
                 \frac{1}{\overline{w} z}\log\left(\frac{1}{1-\overline{w} z}\right),&\hbox{for }\alpha=0.
                \end{cases}
\]
Suppose that a finite positive measure $\nu$ on $\D$ satisfies
\begin{equation}\label{Cthsis-w}
\int_\D |k_w^\alpha|^2 \, d\nu \leq c \|k_w^\alpha\|_\alpha^2, \hspace{1cm} (w \in \D).
\end{equation}
It is not difficult to see that \eqref{Cthsis-w} is equivalent to say that $\nu$ is an $\alpha$-Carleson measure,
i.e.,
\[
\nu(S(I))=\begin{cases}
O(|I|^\alpha),&\hbox{for } 0<\alpha\leq 1;\\
O((\log\frac{e}{|I|})^{-1}),&\hbox{for }\alpha=0,
          \end{cases}
\]
for every arc $I\subset \T$ (see for instance \cite{ABP-2009} or \cite{seip}).
On the other hand, as noted in \cite{marshall-sundberg} (see also \cite{seip}),
not all $\alpha$-Carleson measures are
Carleson measures for $\DD_\alpha$, for $\alpha \in [0,1)$.
In other words, the (RKT) does not hold in $\DD_\alpha$ spaces, $\alpha \in [0,1)$.
\end{Rem}

\section{Reproducing kernel thesis and compactness}
One can ask whether  the property of being a compact Carleson measure (the embedding $\DD(\mu)\hookrightarrow L^2(\nu)$ is compact) could be also characterized by a property on the family of normalized reproducing kernels. Note that we should normalize the reproducing kernels to obtain a weakly null sequence. However, this is not sufficient as shown in the following result.
\begin{Lem}\label{lem:weakly-convergence}
Let $\mu=\sum_{j=1}^n \alpha_j\delta_{\lambda_j}$, $\lambda_j\in\T$, $\alpha_j>0$, let $\hat k_w^\mu=k_w^{\mu}/\|k_w^{\mu}\|_{\mu}$ be the normalized reproducing kernel for $\DD(\mu)$ at $w$, and let $\zeta\in\T$. Then $\hat k_w^\mu$ converges weakly to zero as $w\to\zeta$ if and only if $\zeta\neq\lambda_j$ for $1\leq j\leq n$.
\end{Lem}

\begin{proof}
We easily see that the normalized reproducing kernel $\hat k_w^\mu$ converges weakly to zero as  $w\to\zeta$ if and only if $\hat k_w^\mu$ converges uniformly to zero on the compact subsets of $\D$ as $w\to\zeta$.

First consider the case of one point mass, $\mu=\alpha\delta_\lambda$, $\lambda\in\T$ and $\alpha>0$. Note that
\[
\frac{k_w^{\mu}(z)}{\|k_w^{\mu}\|_{\mu}} = \frac{1-\overline{b_\lambda(w)} b_\lambda(z)}{1-\overline{w}z}
\, \cdot \, \Bigl(\frac{1-|w|^2}{1-|b_\lambda(w)|^2}\Bigr)^{1/2},
\]
where $b_\lambda$ is the function in the unit ball of $H^\infty$ given by equation \eqref{kernel}. Moreover, if $K$ is a compact subset of $\D$ and $z \in K$, we have
\[
\Bigl|\frac{1-\overline{b_\lambda(w)} b_\lambda(z)}{1-\overline{w}z}\Bigr| \asymp 1, \hspace{1cm} (w \in \D).
\]
Then $\hat k_w^\mu$ converges weakly to zero as  $w\to\zeta$ if and only if
\[
\lim_{w \to \zeta} \frac{1-|w|^2}{1-|b_\lambda(w)|^2} = 0,
\]
or equivalently, $b_\lambda$ has no angular derivative at  $\zeta$ (recall that $b_\lambda$ is said to have an angular derivative at point $\zeta$ if $b_\lambda$ and $b'_\lambda$ have angular limits at point $\zeta$ and $|b_\lambda(\zeta)|=1$).  This is true if and only if $\zeta\in\T\setminus\{\lambda\}$.

For the general case,  fix $j\in\{1,\dots,n\}$ and note that  \eqref{eqposdef} can be written as
\[
|\hat k_w^{\mu_j}(z)|\leq |\hat k_w^\mu(z)|\frac{\|k_z^{\mu_j}\|_{\mu_j}}{\|k_z^\mu\|_{\mu}}\qquad (z,w\in\D).
\]
Now take $K$ a compact subset of $\D$. According to \eqref{eqenough2}, there exists a constant $C(K)>0$ such that for any $z\in K$, we have
\[
\frac{\|k_z^{\mu_j}\|_{\mu_j}}{\|k_z^\mu\|_{\mu}}\lesssim \frac{1}{\prod_{i\neq j}|z-\lambda_i|}\lesssim C(K).
\]
Hence
\[
|\hat k_w^{\mu_j}(z)|\lesssim C(K) |\hat k_w^\mu(z)|\qquad (z\in K).
\]
Now if $\hat k_w^\mu$ converges uniformly to zero on $K$ as $w\to\zeta$, the last inequality implies that $\hat k_w^{\mu_j}$ also converges uniformly to zero on $K$ as $w\to\zeta$. The first part of the proof now gives that $\zeta\neq \lambda_j$.

Conversely let $\zeta\neq\lambda_j$, $j=1,\dots,n$, and let $f\in\DD(\mu)$. We have
\[
\langle f,\hat k_w^\mu \rangle_\mu=\frac{1}{\|k_w^\mu\|_{\mu}}f(w)=\frac{\|k_w^{\mu_j}\|_{\mu_j}}{\|k_w^\mu\|_\mu}\langle f,\hat k_w^{\mu_j}\rangle_{\mu_j}.
\]
Using once more \eqref{eqenough2}, we get
\[
|\langle f,\hat k_w^\mu \rangle_\mu|\lesssim \frac{1}{\prod_{i\neq j}|w-\lambda_i|}|\langle f,\hat k_w^{\mu_j}\rangle_{\mu_j}|.
\]
But since $\zeta\neq\lambda_j$ for any $j=1,\dots,n$, the first part of the proof shows that $\langle f,\hat k_w^{\mu_j}\rangle_{\mu_j}\to 0$ and $\prod_{i\neq j}|w-\lambda_i|$ remains bounded below, as $w\to\zeta$. Thus $\langle f,\hat k_w^\mu \rangle_\mu\to 0$ as $w\to\zeta$.

\end{proof}

The following result shows now that the reproducing kernel thesis is true for the compactness of the embedding $\DD(\mu)\hookrightarrow L^2(\nu)$ when $\mu$ is a finite sum of point masses.

\begin{Thm}
Let $\nu$ be a finite positive Borel measure on $\D$, and let $\mu=\sum_{j=1}^n\alpha_j \delta_{\lambda_j}$, where $\lambda_j\in\T$ and $\alpha_j>0$. Then $\nu$ is a compact Carleson measure for $\DD(\mu)$ if and only if for any $\zeta\in\T\setminus\{\lambda_1,\dots,\lambda_n\}$, we have
\begin{equation}\label{C-cthesis-compact}
\int_{\D} |k_w^{\mu}|^2 \, d\nu \,=\, o( \|k_w^{\mu}\|_{\mu}^2), \hspace{1cm} (w \to \zeta).
\end{equation}
\end{Thm}

\begin{proof}
According to Lemma~\ref{lem:weakly-convergence}, condition~\eqref{C-cthesis-compact} is necessary for $\nu$ being a compact Carleson measure for $\DD(\mu)$. Conversely assume that \eqref{C-cthesis-compact} is satisfied for any $\zeta\neq \lambda_j$, $j=1,\dots,n$. According to Theorem \ref{c-Car-m}, we shall show that $d\sigma(z):=\prod_{i=1}^n|z-\lambda_i|^2 d\nu(z)$ is a compact Carleson measure for $H^2$, which is equivalent (by Carleson's result) to
\begin{equation}\label{eq:carleson-condition-compact}
\sigma(S(\zeta,h))=o(h),\qquad h\to 0
\end{equation}
for any $\zeta\in\T$. First let us consider $\zeta\in\T\setminus\{\lambda_1,\dots,\lambda_n\}$ and fix $j\in\{1,\dots,n\}$. Using \eqref{eqposdef} and \eqref{eqenough2}, we have
\begin{eqnarray*}
\frac{1}{\|k_w^\mu\|_\mu^2} \int_{\D} |k_w^\mu(z)|^2d\nu(z)&\geq &\int_\D \frac{|k_w^{\mu_j}(z)|^2\|k_z^\mu\|_\mu^2}{\|k_w^{\mu_j}\|_{\mu_j}^2\|k_z^{\mu_j}\|_{\mu_j} ^2}d\nu(z)\\
&\gtrsim&\frac{1}{\|k_w^{\mu_j}\|_{\mu_j}^2} \int_\D |k_w^{\mu_j}(z)|^2\prod_{i\neq j}|z-\lambda_j|^2d\nu(z).
\end{eqnarray*}
But we know that
\[
\frac{|k_w^{\mu_j}(z)|^2}{\|k_w^{\mu_j}\|_{\mu_j}^2}=\frac{1-|w|^2}{1-|b_{\lambda_j}(w)|^2}\frac{|1-\overline{b_{\lambda_j}(w)}b_{\lambda_j}(z)|^2}{|1-\bar wz|^2},
\]
and according to \eqref{inf-1}, there exists a constant $\delta>0$ such that for any $z,w\in\D$, we have
\[
\frac{|1-\overline{b_{\lambda_j}(w)}b_{\lambda_j}(z)|^2}{1-|b_{\lambda_j}(w)|^2}\geq \delta |z-\lambda_j|^2.
\]
Thus we get
\[
\frac{1}{\|k_w^\mu\|_\mu^2} \int_{\D} |k_w^\mu(z)|^2d\nu(z)\gtrsim (1-|w|^2)\int_\D \frac{1}{|1-\bar w z|^2}d\sigma(z)\]
and condition~\eqref{C-cthesis-compact} gives
\[
\lim_{w\to\zeta}\left((1-|w|^2)\int_\D \frac{1}{|1-\bar w z|^2}d\sigma(z)\right)=0
\]
for any $\zeta\in\T\setminus\{\lambda_1,\dots,\lambda_n\}$. Now a standard argument shows that \eqref{eq:carleson-condition-compact} is satisfied for $\zeta\neq\lambda_i$, $i=1,\dots,n$. Indeed take $w=(1-h)\zeta$. Then, on one hand, for $z\in S(\zeta,h)$, we have
\[
|1-\bar w z|=|1-(1-h)\bar\zeta z|=|\zeta-(1-h)z|\leq |\zeta-z|+h\leq 2h.
\]
And on the other hand,  $1-|w|^2\geq 1-|w|=h$. Then
\[
\frac{1}{4h}\sigma(S(\zeta,h))\leq (1-|w|^2)\int_\D \frac{1}{|1-\overline{w}z|^2}\,d\sigma(z)\to 0\qquad \hbox{as }h\to 0,
\]
which gives \eqref{eq:carleson-condition-compact}. It remains to notice that condition~\eqref{eq:carleson-condition-compact} for $\zeta=\lambda_i$, $i=1,\dots,n$, follows from the following trivial estimate
\[
\sigma(S(\lambda_i,h))=\int_{S(\lambda_i,h)}\prod_{j=1}^n |z-\lambda_j|^2\,d\nu(z)\leq 4^{n-1}h^2\|\nu\|.
\]
\end{proof}

\end{document}